\documentclass[10pt,reqno]{amsart}

\usepackage{amsthm, mathrsfs, amsmath, amstext, amsxtra, amsfonts, dsfont, amssymb, color}
\usepackage{lmodern}
\definecolor{green_dark}{rgb}{0,0.6,0}
\usepackage[colorlinks, linkcolor=red, citecolor=blue, urlcolor=green_dark, pagebackref, hypertexnames=false]{hyperref}

\newcommand{\N}{\mathbb N}

\newcommand{\R}{\mathbb R}
\newcommand{\C}{\mathbb C}

\newcommand{\ep}{\epsilon}
\newcommand{\Gact} {\gamma_{\text{c}}}

\newcommand{\re}[1]{\mbox{Re} \ #1} 
\newcommand{\im}[1]{\mbox{Im} \ #1} 
\newcommand{\reemph}[1]{\mbox{\emph{Re}} \ #1} 
\newcommand{\imemph}[1]{\mbox{\emph{Im}} \ #1}
\newcommand{\scal}[1]{\left\langle #1 \right\rangle} 

\newcommand{\defendproof}{\hfill $\Box$} 

\newtheorem{theorem}{Theorem}[section]
\newtheorem{defi}[theorem]{Definition}
\newtheorem{lem}[theorem]{Lemma} 
\newtheorem{prop}[theorem]{Proposition}
\newtheorem{coro}[theorem]{Corollary} 
\theoremstyle{definition}
\newtheorem{rem}[theorem]{Remark}

\title[Energy scattering for defocusing inhomogeneous NLS]{Energy scattering for a class of the defocusing inhomogeneous nonlinear Schr\"odinger equation} 

\author[V. D. Dinh]{Van Duong Dinh}
\address[V. D. Dinh]{Institut de Math\'ematiques de Toulouse UMR5219, Universit\'e Toulouse CNRS, 31062 Toulouse Cedex 9, France}
\email{dinhvan.duong@math.univ-toulouse.fr}


\begin{document}

\maketitle
\begin{abstract}
In this paper, we consider a class of the defocusing inhomogeneous nonlinear Schr\"odinger equation
\[
i\partial_t u + \Delta u - |x|^{-b} |u|^\alpha u = 0, \quad u(0)=u_0 \in H^1,
\]
with $b, \alpha>0$. We firstly study the decaying property of global solutions for the equation when $0<\alpha<\alpha^\star$ where $\alpha^\star = \frac{4-2b}{d-2}$ for $d\geq 3$. The proof makes use of an argument of Visciglia in \cite{Visciglia}. We next use this decay to show the energy scattering for the equation in the case $\alpha_\star<\alpha<\alpha^\star$, where $\alpha_\star = \frac{4-2b}{d}$.
\end{abstract}


\section{Introduction}
\setcounter{equation}{0}
Consider the Cauchy problem for the inhomogeneous nonlinear Schr\"odinger equation
\[
\left\{
\begin{array}{ccl}
i\partial_t u + \Delta u + \mu |x|^{-b} |u|^\alpha u &=& 0, \\
u(0)&=& u_0,
\end{array} 
\right. \tag{INLS}
\]
where $u: \R \times \R^d \rightarrow \C, u_0:\R^d \rightarrow \C$, $\mu = \pm 1$ and $\alpha, b>0$. The parameters $\mu=1$ and $\mu=-1$ correspond to the focusing and defocusing cases respectively. The case $b=0$ is the well-known nonlinear Schr\"odinger equation which has been studied extensively over the last three decades. The inhomogeneous nonlinear Schr\"odinger equation arises naturally in nonlinear optics for the propagation of laser beams, and it is of a form
\begin{align}
i \partial_t u +\Delta u + K(x) |u|^\alpha u =0. \label{general INLS}
\end{align}
The (INLS) is a particular case of $(\ref{general INLS})$ with $K(x)=|x|^{-b}$. The equation $(\ref{general INLS})$ has been attracted a lot of interest in a past several years. Berg\'e in \cite{Berge} studied formally the stability condition for soliton solutions of $(\ref{general INLS})$. Towers-Malomed in \cite{TowersMalomed} observed by means of variational approximation and direct simulations that a certain type of time-dependent nonlinear medium gives rise to completely stabe beams. Merle in \cite{Merle} and Rapha\"el-Szeftel in \cite{RaphaelSzeftel} studied $(\ref{general INLS})$ for $k_1<K(x)<k_2$ with $k_1, k_2>0$. Fibich-Wang in \cite{FibichWang} investigated $(\ref{general INLS})$ with $K(x):=K(\ep|x|)$  where $\ep>0$ is small and $K \in C^4(\R^d) \cap L^\infty(\R^d)$. The case $K(x) = |x|^{b}$ with $b>0$ is studied by many authors (see e.g. \cite{Chen, LiuWangWang, Zhu} and references therein). \newline
\indent In order to review known results for the (INLS), we recall some facts for this equation.  We firstly note that the (INLS) is invariant under the scaling,
\[
u_\lambda(t,x):= \lambda^{\frac{2-b}{\alpha}} u(\lambda^2 t, \lambda x), \quad \lambda>0.
\]
An easy computation shows
\[
\|u_\lambda(0)\|_{\dot{H}^\gamma(\R^d)} = \lambda^{\gamma+\frac{2-b}{\alpha}-\frac{d}{2}} \|u_0\|_{\dot{H}^\gamma(\R^d)}.
\]
Thus, the critical Sobolev exponent is given by 
\begin{align}
\Gact := \frac{d}{2}-\frac{2-b}{\alpha}. \label{critical exponent}
\end{align}
Moreover, the (INLS) has the following conserved quantities:
\begin{align}
M(u(t))&:= \int_{\R^d} |u(t,x)|^2 dx = M(u_0), \label{mass conservation} \\
E(u(t)) &:= \int_{\R^d} \frac{1}{2}|\nabla u(t,x)|^2 - \frac{\mu}{\alpha+2} |x|^{-b}|u(t,x)|^{\alpha+2} dx = E(u_0). \label{energy conservation} 
\end{align}
\indent The well-posedness for the (INLS) was firstly studied by Genoud-Stuart in \cite[Appendix]{GenoudStuart} by using the argument of Cazenave \cite{Cazenave} which does not use Strichartz estimates. More precisely, the authors showed that the focusing (INLS) with $0<b<\min\{2,d\}$ is well posed in $H^1(\R^d)$: 
\begin{itemize}
\item locally if $0<\alpha<\alpha^\star$,
\item globally for any initial data if $0<\alpha <\alpha_\star$,
\item globally for small initial data if $\alpha_\star \leq \alpha <\alpha^\star$,
\end{itemize}
where $\alpha_\star$ and $\alpha^\star$ are defined by
\begin{align}
\renewcommand*{\arraystretch}{1.2}
\alpha_\star:=\frac{4-2b}{d}, \quad \alpha^\star := \left\{ \begin{array}{c l}
\frac{4-2b}{d-2} & \text{if } d\geq 3, \\
\infty &\text{if } d=1, 2.
\end{array} \right. \label{alpha exponents}
\end{align}
\indent In the case $\alpha=\alpha_\star$ ($L^2$-critical), Genoud in \cite{Genoud} showed that the focusing (INLS) with $0<b<\min\{2,d\}$ is globally well-posed in $H^1(\R^d)$ assuming $u_0 \in H^1(\R^d)$ and
\[
\|u_0\|_{L^2(\R^d)} <\|Q\|_{L^2(\R^d)},
\]
where $Q$ is the unique nonnegative, radially symmetric, decreasing solution of the ground state equation
\[
\Delta Q -Q +|x|^{-b}|Q|^{\frac{4-2b}{d}} Q=0.
\]
Also, Combet-Genoud in \cite{CombetGenoud} established the classification of minimal mass blow-up solutions for the focusing $L^2$-critical (INLS). \newline
\indent In the case $\alpha_\star <\alpha<\alpha^\star$, Farah in \cite{Farah} showed that the focusing (INLS) with $0<b<\min \{2,d\}$ is globally well-posedness in $H^1(\R^d), d\geq 3$ assuming $u_0 \in H^1(\R^d)$ and
\begin{align}
E(u_0)^{\Gact} M(u_0)^{1-\Gact} &< E(Q)^{\Gact} M(Q)^{1-\Gact}, \label{condition 1}\\
\|\nabla u_0\|^{\Gact}_{L^2(\R^d)} \|u_0\|_{L^2(\R^d)}^{1-\Gact} &< \|\nabla Q\|^{\Gact}_{L^2(\R^d)} \|Q\|^{1-\Gact}_{L^2(\R^d)}, \nonumber
\end{align}
where $Q$ is the unique nonnegative, radially symmetric, decreasing solution of the ground state equation
\[
\Delta Q- Q + |x|^{-b} |Q|^\alpha Q =0.
\]
He also proved that if $u_0 \in H^1(\R^d) \cap L^2(\R^d, |x|^2dx)=:\Sigma$ satisfies $(\ref{condition 1})$ and 
\begin{align}
\|\nabla u_0\|^{\Gact}_{L^2(\R^d)} \|u_0\|_{L^2(\R^d)}^{1-\Gact} &> \|\nabla Q\|^{\Gact}_{L^2(\R^d)} \|Q\|^{1-\Gact}_{L^2(\R^d)},
\label{condition 2}
\end{align}
then the blow-up in $H^1(\R^d)$ must occur. Afterwards, Farah-Guzman in \cite{FarahGuzman1, FarahGuzman2} proved that the above global solution is scattering under the radial condition of the initial data.  \newline
\indent Recently, Guzman in \cite{Guzman} used Strichartz estimates and the contraction mapping argument to establish the well-posedness for the (INLS) in Sobolev space. Precisely, he showed that:
\begin{itemize}
\item if $0<\alpha<\alpha_\star$ and $0<b<\min\{2, d\}$, then the (INLS) is locally well-posed in $L^2(\R^d)$. It is then globally well-posed in $L^2(\R^d)$ by mass conservation.
\item if $0<\alpha<\widetilde{\alpha}, 0<b<\widetilde{b}$ and $\max\{0, \Gact\}<\gamma \leq \min\left\{\frac{d}{2},1 \right\}$ where 
\begin{align}
\renewcommand*{\arraystretch}{1.2}
\widetilde{\alpha}:= \left\{ \begin{array}{c l}
\frac{4-2b}{d-2\gamma} & \text{if } \gamma<\frac{d}{2}, \\
\infty &\text{if } \gamma=\frac{d}{2}
\end{array} \right. \quad \text{and} \quad  \widetilde{b}:= \left\{ \begin{array}{c l}
\frac{d}{3} & \text{if } d=1, 2, 3, \\
2 &\text{if } d\geq 4,
\end{array} \right. \label{define tilde alpha and b}
\end{align}
then the (INLS) is locally well-posed in $H^\gamma(\R^d)$. 
\item if $\alpha_\star<\alpha<\widetilde{\alpha}$, $0<b<\widetilde{b}$ and $\Gact<\gamma\leq \min\left\{\frac{d}{2}, 1\right\}$, then the (INLS) is globally well-posed in $H^\gamma(\R^d)$ for small initial data. 
\end{itemize}
In particular, Guzman proved the following local well-posedness in the energy space for the (INLS).
\begin{theorem}[\cite{Guzman}] \label{theorem guzman}
Let $d\geq 2, 0<b<\widetilde{b}$ and $0<\alpha<\alpha^\star$. Then the \emph{(INLS)} is locally well-posed in $H^1(\R^d)$. Moreover, the local solutions satisfy $u \in L^p_{\emph{loc}}(\R, L^q(\R^d))$ for any Schr\"odinger admissible pair $(p,q)$. 
\end{theorem}
Recently, the author in \cite{Dinhweighted} improved the range of $b$ in Theorem $\ref{theorem guzman}$ in the two and three dimensional spatial spaces. More precisely, he proved the following:
\begin{theorem}[\cite{Dinhweighted}] \label{theorem dinhweighted}
Let 
\[
d\geq 4, \quad 0<b<2, \quad 0 <\alpha <\alpha^\star,
\]
or 
\[
d=3, \quad 0< b <1, \quad 0<\alpha <\alpha^\star,
\]
or 
\[
d=3, \quad 1 \leq b<\frac{3}{2}, \quad 0<\alpha <\frac{6-4b}{2b-1},
\]
or
\[
d=2, \quad 0<b<1, \quad 0 <\alpha< \alpha^\star.
\]
Then the \emph{(INLS)} is locally well-posed in $H^1(\R^d)$. Moreover, the solutions satisfy $u \in L^p_{\emph{loc}}(\R, L^q(\R^d))$ for any Schr\"odinger admissible pair $(p,q)$.
\end{theorem}
Note that the results of Guzman \cite{Guzman} and Dinh \cite{Dinhweighted} about the local well-posedness of (INLS) in $H^1(\R^d)$ are a bit weaker than the one of Genoud-Stuart \cite{GenoudStuart}. Precisely, they do not treat the case $d=1$, and there is a restriction on the validity of $b$ when $d=2$ or $3$. Note also that the author in \cite{Dinhweighted} pointed out that one cannot expect a similar result as Theorem $\ref{theorem guzman}$ or Theorem $\ref{theorem dinhweighted}$ holds in the one dimensional case by using Strichartz estimates. Although the result showed by Genoud-Stuart is strong, but one does not know whether the local solutions belong to $L^p_{\text{loc}}(\R,L^q(\R^d))$ for any Schr\"odinger admissible pair $(p,q)$. This property plays an important role in proving the scattering for the defocusing (INLS). \newline
\indent Note that the local well-posedness (which is also available for the defocusing case) of Genoud-Stuart in \cite{GenoudStuart} and the conservations of mass and energy immediately give the global well-posedness in $H^1(\R^d)$ for the defocusing (INLS). In \cite{Dinhweighted}, the author used the pseudo-conformal conservation law to show the decaying property of global solutions by assuming the initial data in $\Sigma$ (see before $(\ref{condition 2})$). In particular, he showed that in the case $\alpha \in [\alpha_\star, \alpha^\star)$, global solutions have the same decay as the solutions of the linear Schr\"odinger equation, that is for $2\leq q \leq \frac{2d}{d-2}$ when $d\geq 3$ or $2\leq q <\infty$ when $d=2$ or $2\leq q\leq\infty$ when $d=1$,
\[
\|u(t)\|_{L^q(\R^d)} \lesssim |t|^{-d\left(\frac{1}{2}-\frac{1}{q}\right)}, \quad \forall t\ne 0. 
\]
This allows the author proved the scattering in $\Sigma$ for a certain class of the defocusing (INLS). We refer the reader to \cite{Dinhweighted} for more details.  \newline
\indent The main purpose of this paper is to show the energy scattering property for the defocusing (INLS). Before stating our resuts, let us recall some known 	techniques to prove the energy scattering for the nonlinear Schr\"odinger equation (NLS). To our knowledge, there are two methods to prove the energy scattering for the (NLS). The first one is to use the classical Morawetz inequality to derive the decay of global solutions, and then use it to prove the global Strichartz bound of solutions (see e.g. \cite{GinibreVelo, Nakanishi} or \cite{Cazenave}). The second one is to use the interaction Morawetz inequality to derive directly the global Strichartz bound for solutions (see e.g. \cite{TaoVisanZhang}, \cite{CollianderGrillakisTzirakis} and references therein). With the global Strichartz bound at hand, the energy scattering follows easily. Note also that Visciglia in \cite{Visciglia} used the interaction Morawetz inequality to show the decay of global solutions. This allows the author to show the decaying property for the (NLS) in any dimensions. This approach is a complement to \cite{GinibreVelo} where the classical Morawetz inequality only allowed to prove the decaying property in spatial dimensions greater than or equal to three. It is worth noticing that the (INLS) does not enjoy the conservation of momentum which is crucial to prove the interaction Morawetz-type inequality (see e.g. \cite{CollianderGrillakisTzirakis}). We thus do not attempt to show the interaction Morawetz-type inequality for the defocusing (INLS). It is also not clear to us that the techniques of \cite{GinibreVelo, Nakanishi} can be applied for the defocusing (INLS). Fortunately, we are able to use the classical Morawetz-type inequality and the technique of \cite{Visciglia} to show the decay of global solutions for the defocusing (INLS). More precisely, we have the following decay of global solutions to the defocusing(INLS). 
\begin{theorem} \label{theorem decay property}
Let $d\geq 3, 0<b<2$ and $0<\alpha <\alpha^\star$. Let $u_0 \in H^1(\R^d)$ and $u \in C(\R, H^1(\R^d))$ be the unique global solution to the defocusing \emph{(INLS)}. Then,
\begin{align}
\lim_{t\rightarrow \pm \infty} \|u(t)\|_{L^q(\R^d)} =0, \label{decay property}
\end{align}
for every $q \in (2, 2^\star)$, where $2^\star:=\frac{2d}{d-2}$.
\end{theorem}
The proof of this result is based on the classical Morawetz-type inequality and an argument of Visciglia in \cite{Visciglia}. The classical Morawetz-type inequality related to the defocusing (INLS) is derived by using the same argument of that for the classical (NLS). This inequality is enough to prove the decaying property for global solutions of the defocusing (INLS) by following the technique of \cite{Visciglia}. Note that in \cite{Visciglia}, the author used the interaction Morawetz inequality to show the decay of solutions for the defocusing (NLS) in any dimensions. We expect that the decay $(\ref{decay property})$ still holds in dimensions 1 and 2. But it is not clear to us how to prove it at the moment. \newline
\indent Using the decaying property given in Theorem $\ref{theorem decay property}$, we are able to show the energy scattering for the defocusing (INLS). Due to the singularity of $|x|^{-b}$, the scattering result does not cover the full range of exponents as in Theorem $\ref{theorem dinhweighted}$. Our main result is the following:
\begin{theorem} \label{theorem energy scattering}
Let 
\[
d\geq 4, \quad 0<b<2, \quad \alpha_\star <\alpha <\alpha^\star,
\]
or 
\[
d=3, \quad 0< b <\frac{5}{4}, \quad \alpha_\star<\alpha <3-2b.
\] 
Let $u_0 \in H^1(\R^d)$ and $u$ be the unique global solution to the defocusing \emph{(INLS)}. Then there exists $u_0^\pm \in H^1(\R^d)$ such that
\[
\lim_{t\rightarrow \pm \infty} \|u(t)-e^{it\Delta} u_0^\pm\|_{H^1(\R^d)} =0.
\]
\end{theorem}
The proof of this result is based on a standard argument as for the nonlinear Schr\"odinger equation (see e.g. \cite[Chapter 7]{Cazenave}). Because of the singularity $|x|^{-b}$, one needs to be careful in order to control the nonlinearity in terms of decay norms and Strichartz norms. The singularity also leads to a restriction on the ranges of $b$ and $\alpha$ compared to those in Theorem $\ref{theorem decay property}$. We expect that the same result still holds true in the two dimensional case. This expectation will be possible if one can show the same decay as in Theorem $\ref{theorem decay property}$ in 2D. \newline
\indent The plan of this paper is as follows. In Section 2, we introduce some notations and give some preliminary results related to our problem. In Section 3, we derive classical Morawetz-type inequalities for the defocusing (INLS). The proof of the decaying property of Theorem $\ref{theorem decay property}$ is given in Section 4. Section 5 is devoted to the proof of the scattering result of Theorem $\ref{theorem energy scattering}$.
\section{Preliminaries} \label{section preliminaries}
\setcounter{equation}{0}
In the sequel, the notation $A \lesssim B$ denotes an estimate of the form $A\leq CB$ for some constant $C>0$. The constant $C$ may change from line to line. 
\subsection{Nonlinearity} \label{subsection nonlinearity}
Let $F(x, z):=|x|^{-b} f(z)$ with $b>0$ and $f(z):=|z|^\alpha z$. The complex derivatives of $f$ are
\[
\partial_zf(z) = \frac{\alpha+2}{2}|z|^\alpha, \quad \partial_{\overline{z}} f(z) = \frac{\alpha}{2} |z|^{\alpha-2} z^2.
\] 
We have for $z, w \in \C$, 
\[
f(z) - f(w) = \int_0^1 \Big(\partial_z f(w+\theta(z-w)) (z-w) + \partial_{\overline{z}} f(w+\theta(z-w)) \overline{z-w} \Big) d\theta.
\]
Thus,
\begin{align}
|F(x, z)-F(x, w)| \lesssim |x|^{-b} (|z|^\alpha+|w|^\alpha) |z-w|. \label{nonlinear inequality}
\end{align}
To deal with the singularity $|x|^{-b}$, we have the following remark.
\begin{rem}[\cite{Guzman}] \label{rem dealing singularity}
Let $B:=B(0,1)=\{x\in \R^d : |x|<1\}$ and $B^c:=\R^d\backslash B$. Then
\[
\||x|^{-b}\|_{L^\gamma(B)} <\infty \quad \text{if} \quad \frac{d}{\gamma}>b, 
\]
and
\[
\||x|^{-b}\|_{L^\gamma(B^c)} <\infty \quad \text{if} \quad \frac{d}{\gamma}<b. 
\]
\end{rem}
\subsection{Strichartz estimates} \label{subsection strichartz estimates}
Let $J \subset \R$ and $p, q \in [1,\infty]$. We define the mixed norm
\[
\|u\|_{L^p_t(J, L^q_x)} := \Big( \int_J \Big( \int_{\R^d} |u(t,x)|^q dx \Big)^{\frac{1}{q}} \Big)^{\frac{1}{p}}
\] 
with a usual modification when either $p$ or $q$ are infinity. When there is no risk of confusion, we may write $L^p_t L^q_x$ instead of $L^p_t(J,L^q_x)$. We also use $L^p_{t,x}$ when $p=q$.
\begin{defi}
A pair $(p,q)$ is said to be \textbf{Schr\"odinger admissible}, for short $(p,q) \in S$, if 
\[
(p,q) \in [2,\infty]^2, \quad (p,q,d) \ne (2,\infty,2), \quad \frac{2}{p}+\frac{d}{q} = \frac{d}{2}.
\]
\end{defi}
We denote for any spacetime slab $J\times \R^d$,
\begin{align}
\|u\|_{S(L^2, J)}:= \sup_{(p,q) \in S} \|u\|_{L^p_t(J,L^q_x)}, \quad \|v\|_{S'(L^2,J)}:=\inf_{(p,q)\in S} \|v\|_{L^{p'}_t(J, L^{q'}_x)}. \label{define strichartz norm}
\end{align}
We next recall well-known Strichartz estimates for the linear Schr\"odinger equation. We refer the reader to \cite{Cazenave, Tao} for more details.
\begin{prop} \label{prop strichartz esimates}
Let $u$ be a solution to the linear Schr\"odinger equation, namely
\[
u(t)= e^{it\Delta}u_0 + \int_0^t e^{i(t-s)\Delta} F(s) ds,
\]
for some data $u_0, F$. Then,
\begin{align}
\|u\|_{S(L^2,\R)} \lesssim \|u_0\|_{L^2_x} + \|F\|_{S'(L^2, \R)}. \label{strichartz estimates}
\end{align}
\end{prop}
\section{Classical Morawetz-type inequality} \label{section classical morawetz-type inequality}
\setcounter{equation}{0}
In this section, we will derive interaction Morawetz inequalities for the defocusing (INLS) by following the technique of \cite{TaoVisanZhang}. Given a smooth real valued function $a$, we define the Morawetz action by
\begin{align}
M_a(t):= 2 \int_{\R^d}\nabla a(x) \cdot  \im{(\overline{u}(t,x) \nabla u(t,x))} dx. \label{morawetz action} 
\end{align}
By a direct computation, we have the following result.
\begin{lem}[\cite{TaoVisanZhang}]  \label{lem derivative morawetz action}
If $u$ is a smooth-in-time and Schwartz-in-space solution to 
\[
i\partial_t u +\Delta u = N(u),
\]
with $N(u)$ satisfying $\imemph{(N(u) \overline{u})} =0$, then we have
\begin{equation}
\begin{aligned}
\frac{d}{dt} M_a(t) = -\int \Delta^2 a(x) |u(t,x)|^2  dx & +  4 \sum_{j,k=1}^d \int \partial^2_{jk} a(x) \reemph{(\partial_k u(t,x) \partial_j \overline{u}(t,x))} dx  \\
&+ 2\int \nabla a(x)\cdot \{N(u), u\}_p(t,x) dx, 
\end{aligned} \label{derivative morawetz action} 
\end{equation}
where $\{f, g\}_p :=\reemph{(f\nabla \overline{g} - g \nabla \overline{f})}$ is the momentum bracket.
\end{lem} 
We refer the reader to \cite[Lemma 5.3]{TaoVisanZhang} for the proof of this result. Note that if $N(u) =F(x,u)= |x|^{-b} |u|^\alpha u$, then we have \footnote{See the Appendix for the proof.}
\begin{align}
\{N(u), u\}_p = -\frac{\alpha}{\alpha+2} \nabla(|x|^{-b} |u|^{\alpha+2}) -\frac{2}{\alpha+2} \nabla(|x|^{-b}) |u|^{\alpha+2}. \label{momentum bracket}
\end{align}
In particular, we have the following consequence.
\begin{coro}\label{coro derivative morawetz action}
If $u$ is a smooth-in-time and Schwartz-in-space solution to the defocusing \emph{(INLS)}, then we have 
\begin{equation}
\begin{aligned}
\frac{d}{dt} M_a(t) &= -\int \Delta^2 a(x) |u(t,x)|^2  dx  +  4 \sum_{j,k=1}^d \int \partial^2_{jk} a(x) \reemph{(\partial_k u(t,x) \partial_j \overline{u}(t,x))} dx  \\
&\mathrel{\phantom{=}}+\frac{2\alpha}{\alpha+2} \int \Delta a(x) |x|^{-b} |u(t,x)|^{\alpha+2} dx -\frac{4}{\alpha+2} \int \nabla a(x) \cdot \nabla(|x|^{-b}) |u(t,x)|^{\alpha+2} dx. 
\end{aligned} \label{derivative morawetz action INLS} 
\end{equation}
\end{coro}
With the help of Lemma $\ref{lem derivative morawetz action}$, we obtain the following classical Morawetz-type inequalities for the defocusing (INLS). 
\begin{prop} \label{prop interaction morawetz inequality}
Let $d\geq 3, 0<b<2$ and $u$ be a solution to the defocusing \emph{(INLS)} on the spacetime slab $J \times \R^d$. Then 
\begin{align}
\int_J \int_{\R^d} |x|^{-b-1} |u(t,x)|^{\alpha+2} dxdt < \infty. \label{classical morawetz d geq 3}
\end{align}
\end{prop}
\begin{proof} We consider $a(x)=|x|$. An easy computation shows
\begin{align*}
\partial_j a(x)=\frac{x_j}{|x|}, \quad \partial^2_{jk} a(x) = \frac{1}{|x|}\Big(\delta_{jk}-\frac{x_j x_k}{|x|^2} \Big),
\end{align*}
for $j, k =1, \cdots, d$. This implies
\[
\nabla a(x)=\frac{x}{|x|}, \quad \Delta a(x)= \frac{d-1}{|x|},
\] 
and
\[
\renewcommand*{\arraystretch}{1.3}
-\Delta^2 a(x) =  -(d-1)\Delta\Big(\frac{1}{|x|}\Big) =
\left\{
\begin{array}{ll}
4\pi (d-1) \delta_0 & \text{if } d=3, \\
\frac{(d-1)(d-3)}{|x|^3} & \text{if } d\geq 4,
\end{array}
\right.
\]
where $\delta_0$ is the Dirac delta function. Since $a$ is a convex function, it is well-known that
\[
\sum_{j,k=1}^{d}\partial^2_{jk} a \re{(\partial_k u \partial_j \overline{u})} \geq 0.
\]
Therefore, applying $(\ref{derivative morawetz action INLS})$ with $a(x)=|x|$, we get
\[
\frac{d}{dt} M_{|x|}(t) \geq \frac{2\alpha(d-1)+4b}{\alpha+2} \int |x|^{-b-1}|u(t,x)|^{\alpha+2} dx.
\]
Thus,
\[
\int_J\int_{\R^d} |x|^{-b-1}|u(t,x)|^{\alpha+2} dxdt \lesssim \sup_{t\in J} |M_{|x|}(t)| \lesssim \|u(t)\|_{L^\infty_t(J, L^2_x)} \|\nabla u(t)\|_{L^\infty_t(J, L^2_x)} <\infty.
\]
The last estimate follows from the conservations of mass and energy. 
\end{proof}
\begin{rem}
The above method breaks down for $d \leq 2$ since the distribution $-\Delta^2 (|x|)$ is not positive anymore. In this case, one can adapt an argument of Nakanishi in \cite{Nakanishi} to show 
\begin{align}
\int_J \int_{\R^d} \frac{t^2}{(t^2+|x|^2)^{\frac{3}{2}}} |x|^{-b}|u(t,x)|^{\alpha+2} dx dt<\infty. \label{Nakanishi Morawetz estimate}
\end{align}
However, we do not know whether the estimate $(\ref{Nakanishi Morawetz estimate})$ is sufficient to prove the decay of global solutions to the defocusing (INLS).
\end{rem}
\section{Decay of global solutions} \label{section decay global solutions}
\setcounter{equation}{0}
In this section, we will give the proof of Theorem $\ref{theorem decay property}$. To do so, we follow the argument of Visciglia in \cite{Visciglia}. Let us start with the following:
\begin{lem} \label{lem approximation solution}
Let $d\geq 3, 0<b<2$ and $0<\alpha<\alpha^\star$. Let $\chi \in C^\infty_0$ be a cutoff function and $\psi_n \in H^1_x$ be a sequence such that 
\[
\sup_{n\in \N} \|\psi_n\|_{H^1_x} <\infty, \quad \text{and} \quad \psi_n \rightharpoonup \psi \text{ weakly in } H^1_x. 
\]
Let $v_n$ and $v \in C(\R, H^1_x)$ be the corresponding solutions to the defocusing \emph{(INLS)} with initial data $\psi_n$ and $\psi$ respectively. Then for every $\ep>0$, there exists $T(\ep)>0$ and $n(\ep) \in \N$ such that 
\begin{align}
\sup_{t\in (0, T(\ep))} \|\chi (v_n(t) -v(t)) \|_{L^2_x} \leq \ep, \quad \forall n > n(\ep). \label{approximation solution}
\end{align}
\end{lem}
\begin{proof}
By the conservations of mass and energy, 
\begin{align}
\sup_{t\in \R, n \in \N} \left\{ \|v_n(t)\|_{H^1_x}, \|v(t)\|_{H^1_x} \right\} <\infty. \label{bounded solutions}
\end{align}
By Rellich's compactness lemma, up to a subsequence, 
\begin{align}
\lim_{n\rightarrow \infty} \|\chi(\psi_n -\psi)\|_{L^2_x}=0. \label{limit initial data}
\end{align}
Now let $w_n(t,x):=\chi(x)v_n(t,x)$ and $w(t,x):= \chi(x)v(t,x)$. It is easy to see that
\[
i \partial_t w_n =-\Delta w_n + 2\nabla \chi \cdot \nabla v_n + v_n \Delta \chi + \chi |x|^{-b} |v_n|^\alpha v_n, \quad w_n(0)= \chi \psi_n,
\]
and 
\[
i \partial_t w = -\Delta w  + 2\nabla \chi \cdot \nabla v + v \Delta \chi + \chi |x|^{-b} |v|^\alpha v, \quad w(0)= \chi \psi. 
\]
Thus, by Duhamel formula,
\begin{align}
w_n(t)-w(t) &= e^{it\Delta} (\chi (\psi_n-\psi)) -i \int_0^t e^{i(t-s)\Delta} \Big( 2\nabla \chi \cdot\nabla(v_n(s)-v(s)) + (v_n(s)-v(s)) \Delta \chi \Big) ds \nonumber \\
&\mathrel{\phantom{=}}  -i \int_0^t e^{i(t-s)\Delta} \Big(\chi |x|^{-b} (|v_n(s)|^\alpha v_n(s) -|v(s)|^\alpha v(s)) \Big)ds. \label{duhamel formula}
\end{align}
Due to the singularity of $|x|^{-b}$, we need to consider two cases: \newline
\indent \textbf{\underline{Case 1:} The support of $\chi$ does not contain the origin.} In this case, the proof follows as in \cite[Lemma 1.1]{Visciglia}. For reader's convenience, we recall some details. Denote $J=(0,T)$. Let us introduce the following Schr\"odinger admissible pair $(p,q)$, where
\[
p=\frac{8}{(d-2)\alpha}, \quad q=\frac{4d}{2d-(d-2)\alpha}.
\]
Using Strichartz estimates, we get
\begin{align}
\|w_n-w\|_{L^p_t(J, L^q_x)} &\lesssim \|\chi(\psi_n-\psi)\|_{L^2_x} + \|\nabla \chi \cdot \nabla(v_n-v)\|_{L^1_t(J, L^2_x)} \nonumber \\
& \mathrel{\phantom{=}} + \|(v_n-v) \Delta \chi\|_{L^1_t(J, L^2_x)} + \|\chi|x|^{-b} (|v_n|^\alpha v_n - |v|^\alpha v )\|_{L^{p'}_t(J, L^{q'}_x)}. \label{application strichartz estimate}
\end{align}
We use $(\ref{bounded solutions})$, H\"older's inequality and the Sobolev embedding $H^1_x \subset L^{\frac{2d}{d-2}}_x$ to get
\begin{align*}
\|w_n-w\|_{L^p_t(J, L^q_x)} &\lesssim \|\chi(\psi_n-\psi)\|_{L^2_x} + |J| + \|\chi(v_n-v)\|_{L^{p'}_t(J, L^q_x)} \sup_{t\in J}\Big(\|v_n(t)\|^\alpha_{L^{\frac{2d}{d-2}}_x} + \|v(t)\|^\alpha_{L^{\frac{2d}{d-2}}_x} \Big) \\
&\lesssim \|\chi(\psi_n-\psi)\|_{L^2_x} + |J| + |J|^{1-\frac{2}{p}}\|v_n-v\|_{L^p_t(J, L^q_x)}.
\end{align*}
We learn from the above estimate and $(\ref{limit initial data})$ that for every $\ep>0$, there exists $n(\ep) \in \N$ and $T(\ep)>0$  such that 
\begin{align}
\|w_n -w\|_{L^p_t(I(\ep), L^q_x)} \leq \ep, \label{bounded Lp Lq}
\end{align}
for all $n > n(\ep)$, where $I(\ep)=(0,T(\ep))$. By applying again Strichartz estimate and arguing as above, we obtain
\[
\|w_n-w\|_{L^\infty_t(I(\ep), L^2_x)} \lesssim \|\chi(\psi_n -\psi)\|_{L^2_x}+ C |I(\ep)| + |I(\ep)|^{1-\frac{2}{p}}\|w_n-w\|_{L^p_t(I(\ep), L^q_x)}.
\]
Combining this estimate with $(\ref{bounded solutions})$ and $(\ref{bounded Lp Lq})$, we prove $(\ref{approximation solution})$. \newline
\indent \textbf{\underline{Case 2:} The support of $\chi$ contains the origin.} Without loss of generality, we assume that $\text{supp}(\chi) \subset B$, where $B$ is the ball centered at the origin and of radius 1. Since we are considering $0<\alpha<\frac{4-2b}{d-2}$, there exists $\delta \in \Big(0, \frac{2-b}{2(d-2)}\Big)$ such that $\alpha=\frac{4-2b}{d-2}-4\delta$. Let us choose a Schr\"odinger admissible pair $(p,q)$ with
\[
p=\frac{4}{2-(d-2)\delta}, \quad q= \frac{2d}{(d-2)(1+\delta)}.
\]
In the view of $(\ref{application strichartz estimate})$, it suffices to bound $\|\chi|x|^{-b} (|v_n|^\alpha v_n - |v|^\alpha v)\|_{L^{p'}_t(J, L^{q'}_x)}$. To do this, we use H\"older's inequality, Sobolev embedding and $(\ref{bounded solutions})$ to get
\begin{align}
\|\chi|x|^{-b} (|v_n|^\alpha v_n - |v|^\alpha v)\|_{L^{p'}_t(J, L^{q'}_x)} & \leq \||x|^{-b} (|v_n|^\alpha v_n - |v|^\alpha v)\|_{L^{p'}_t(J, L^{q'}_x(B))} \nonumber \\
&\lesssim \||x|^{-b}\|_{L^\gamma_x(B)} \||v_n|^\alpha v_n - |v|^\alpha v\|_{L^{p'}_t(J,L^r_x)} \nonumber \\
&\lesssim \|v_n-v\|_{L^{p'}_t(J, L^q_x)} \sup_{t\in I} \Big(\|v_n(t)\|^\alpha_{L^{\frac{2d}{d-2}}_x} + \|v(t)\|^\alpha_{L^{\frac{2d}{d-2}}_x} \Big)  \nonumber \\
&\lesssim |J|^{\frac{(d-2)\delta}{2}} \|v_n-v\|_{L^p_t(J, L^q_x)},   \label{estimate case 2}
\end{align}
where
\[
\gamma=\frac{d}{(d-2)\delta +b}, \quad r = \frac{2d}{4-2b + (d-2)(1-3\delta)}.
\]
By Remark $\ref{rem dealing singularity}$, $\||x|^{-b}\|_{L^\gamma_x(B)} <\infty$ provided $\frac{d}{\gamma}>b$, and it is easy to check that 
\[
\frac{d}{\gamma}= (d-2)\delta + b >b.
\]
With $(\ref{estimate case 2})$ at hand, we argue as in Case 1 to have $(\ref{approximation solution})$. 
\end{proof}
\begin{rem} \label{rem approximation solution}
It is not hard to check that Lemma $\ref{lem approximation solution}$ still holds true for any $d\geq 1, 0<b<\min\{2,d\}$ and $0<\alpha <\alpha^\star$. 
\end{rem}
We are now able to prove the decay of global solutions to the defocusing (INLS).
\paragraph{\bf Proof of Theorem $\ref{theorem decay property}$.} 
We only consider the case $t\rightarrow +\infty$, the case $t\rightarrow - \infty$ is treated similarly. We firstly note that by interpolating between $L^2_x$-norm, $L^{2^\star}_x$-norm and $L^q_x$ with $2<q<2^\star$, it suffices to prove $(\ref{decay property})$ for $q=2+\frac{4}{d}$. We next recall the following localized Gagliardo-Nirenberg inequality for $d \geq 3$, that is
\begin{align}
\|\varphi\|^{2+\frac{4}{d}}_{L^{2+\frac{4}{d}}_x} \leq C \Big( \sup_{x \in \R^d} \|\varphi\|_{L^2(Q_1(x))} \Big)^{\frac{4}{d}} \|\varphi\|^2_{H^1_x}, \label{localized gagliardo nirenberg inequality}
\end{align}
where $Q_r(x)$ is the cubic in $\R^d$ centered at $x$ whose edge has length $r$. Let $u$ be the global solution to the defocusing (INLS). The conservations of mass and energy show that
\[
\sup_{t\in \R} \|u(t)\|_{H^1_x} <\infty. 
\]
Assume by the absurd that there is a sequence $t_n \rightarrow \infty$ such that
\begin{align}
\|u(t_n)\|_{L^{2+\frac{4}{d}}_x} \geq \ep_0>0, \label{contradiction assumption}
\end{align}
for all $n\in \N$. By applying $(\ref{localized gagliardo nirenberg inequality})$ with $\varphi \equiv u(t_n,x)$, we see from $(\ref{contradiction assumption})$ that there exists a sequence $(x_n)_{n\in \N}$ of $\R^d$ such that
\begin{align}
\|u(t_n)\|_{L^2(Q_1(x_n))} \geq \ep_1 >0, \label{lower bound}
\end{align}
for all $n\in \N$. We now set $\psi_n(t,x):= u(t_n, x+x_n)$. By the conservations of mass and energy, 
\[
\sup_{n\in \N} \|\psi_n\|_{H^1_x} <\infty.
\]
Thus, up to a subsequence, there exists $\psi \in H^1$ such that $\psi_n \rightharpoonup \psi$ weakly in $H^1_x$. By Rellich's compactness lemma, up to a subsequence, we have
\begin{align}
\lim_{n \rightarrow \infty} \|\psi_n -\psi\|_{L^2(Q_1(0))} =0. \label{weak convergence limit}
\end{align}
We also have from $(\ref{lower bound})$ that $\|\psi_n\|_{L^2(Q_1(0))} \geq \ep_1$. Thus, $(\ref{weak convergence limit})$ ensures that there exists a positive real number still denoted by $\ep_1$ such that 
\begin{align}
\|\psi\|_{L^2(Q_1(0))} \geq \ep_1. \label{lower bound psi}
\end{align}
Let us now introduce $v_n(t,x)$ and $v(t,x)$ as the solutions to
\[
\left\{\begin{array}{rcl}
i\partial_t v_n + \Delta v_n -|x-x_n|^{-b}|v_n|^\alpha v_n &=&0, \\
v_n(0) &=&\psi_n,
\end{array}
\right.
\]
and 
\[
\left\{\begin{array}{rcl}
i\partial_t v + \Delta v -|x-x_n|^{-b}|v|^\alpha v &=&0, \\
v(0) &=&\psi,
\end{array}
\right.
\]
Let $\chi$ be any cutoff function supported in $Q_2(0)$ such that $\chi \equiv 1$ on $Q_1(0)$. We have from $(\ref{lower bound psi})$ and a continuity argument that there exists $T_1>0$ such that
\[
\inf_{t\in (0, T_1)} \|\chi v(t)\|_{L^2_x} \geq \frac{\ep_1}{2}.
\]
Next, applying Lemma $\ref{lem approximation solution}$, there exists $T_2>0$ and $N \in \N$ such that
\[
\sup_{t\in (0,T_2)} \|\chi (v_n(t)-v(t)\|_{L^2_x} \leq \frac{\ep_1}{4}, 
\]
for all $n >N$. Thus, we get for all $t\in (0,T_0)$ with $T_0=\min\{T_1, T_2\}$ and all $n>N$,
\[
\|\chi v_n(t)\|_{L^2_x} \geq \|\chi v(t)\|_{L^2_x} - \|\chi(v_n(t)-v(t)\|_{L^2_x} \geq \frac{\ep_1}{4}.
\]
By the choice of $\chi$, we have for all $t\in (0,T_0)$ and all $n>N$,
\begin{align}
\|v_n(t)\|_{L^2(Q_2(0))} \geq \frac{\ep_1}{4}. \label{lower bound v_n} 
\end{align}
By the uniqueness of local solution to the (INLS), 
\[
v_n(t,x)=u(t+t_n, x+x_n).
\]
Thus, for all $t\in (t_n,t_n+T_0)$ and all $n>N$,
\begin{align}
\|u(t)\|_{L^2(Q_2(x_n))} \geq \frac{\ep_1}{4}. \label{lower bound u}
\end{align}
Moreover, as $\lim_{n\rightarrow \infty} t_n =+\infty$, we can suppose \footnote{One can reduce the value of $T_0$ and increase the value of $N$ if necessary.} that $t_{n+1}-t_n > T_0$ for $n>N$. By H\"older's inequality,
\begin{align}
 \|u(t)\|_{L^{\alpha+2}(Q_2(x_n))} \gtrsim \|u(t)\|_{L^2(Q_2(x_n))} \geq \frac{\ep_1}{4}, \label{holder bound cubic}
\end{align}
for all $t\in (t_n, t_n+T_0)$ and all $n>N$. \newline
\indent The classical Morawetz inequality $(\ref{classical morawetz d geq 3})$ combined with $(\ref{holder bound cubic})$ imply
\begin{align*}
\infty&> \int_0^{+\infty} \int_{\R^d} |x|^{-b-1}|u(t,x)|^{\alpha+2} dx dt \\
&\gtrsim \sum_{n>N} \int_{t_n}^{t_n+T_0} \int_{Q_2(x_n)} |u(t,x)|^{\alpha+2} dx dt \\
& \gtrsim \sum_{n>N}\Big(\frac{\ep_1}{4}\Big)^{\alpha+2} T_0 =\infty.
\end{align*}
This is impossible, and the proof is complete. 
\defendproof
\section{Scattering property} \label{section scattering property}
\setcounter{equation}{0}
In this section, we will give the proof of the scattering property given in Theorem $\ref{theorem energy scattering}$. To do this, we use Strichartz estimates and the decay property given in Theorem $\ref{theorem decay property}$ to obtain a global bound on the solution. The scattering property follows easily from the standard argument. 
\begin{lem} \label{lem decay property approach} Let $d, b$ and $\alpha$ be as in Theorem $\ref{theorem energy scattering}$. Let $u$ be a solution to the defocusing \emph{(INLS)} on a spacetime slab $J \times \R^d$ and $t_0 \in J$. Then there exists $\theta_1, \theta_2 \in (0, \alpha)$ and $q_1, q_2 \in (2, 2^\star)$ such that
\[
\|u - e^{it\Delta} u(t_0)\|_{S(J)} \lesssim \|u\|^{\alpha-\theta_1}_{L^\infty_t(J, L^{q_1}_x)} \|u\|_{S(J)}^{1+\theta_1}+ \|u\|^{\alpha-\theta_2}_{L^\infty_t(J, L^{q_2}_x)} \|u\|_{S(J)}^{1+\theta_2},
\]
where $\|u\|_{S(J)}:= \|\scal{\nabla} u\|_{S(L^2, J)}$.
\end{lem}
\begin{proof}
By Duhamel's formula, the solution to the defocusing (INLS) can be writen as
\[
u(t) = e^{it\Delta}u(t_0) -i\int_{t_0}^t e^{i(t-s)\Delta} |x|^{-b} |u(s)|^\alpha u(s) ds.
\]
The Strichartz estimate $(\ref{strichartz estimates})$ implies
\[
\|u-e^{it\Delta} u(t_0)\|_{S(J)} \lesssim \||x|^{-b} |u|^\alpha u\|_{S'(L^2,J)} +\|\nabla(|x|^{-b}|u|^\alpha u)\|_{S'(L^2, J)}.
\]
We next bound
\begin{align*}
\||x|^{-b} |u|^\alpha u\|_{S'(L^2,J)} &\leq \||x|^{-b} |u|^\alpha u\|_{S'(L^2(B),J)} + \||x|^{-b} |u|^\alpha u\|_{S'(L^2(B^c),J)}=: A_1 + A_2, \\
\|\nabla(|x|^{-b}|u|^\alpha u)\|_{S'(L^2, J)}&\leq \|\nabla(|x|^{-b}|u|^\alpha u)\|_{S'(L^2(B), J)} + \|\nabla(|x|^{-b}|u|^\alpha u)\|_{S'(L^2(B^c), J)} =:B_1 +B_2.
\end{align*}
\indent \textbf{\underline{On $B$.}} By H\"older's inequality and Remark $\ref{rem dealing singularity}$,
\begin{align*}
A_1 \leq \||x|^{-b}|u|^\alpha u\|_{L^{p_1'}_t (J, L^{q_1'}_x(B))} &\lesssim \||x|^{-b}\|_{L^{\gamma_1}_x(B)} \||u|^\alpha u\|_{L^{p_1'}_t(J, L^{\upsilon_1}_x)} \\
&\lesssim \|\|u\|^{\alpha+1}_{L^{q_1}_x}\|_{L^{p_1'}_t(J)} \\
&\lesssim \|u\|^{\alpha-\theta_1}_{L^\infty_t(J, L^{q_1}_x)} \|u\|^{1+\theta_1}_{L^{p_1}_t(J, L^{q_1}_x)},
\end{align*}
provided that $(p_1, q_1)\in S$ and
\[
\frac{1}{q_1'}=\frac{1}{\gamma_1}+\frac{1}{\upsilon_1}, \quad \frac{d}{\gamma_1}>b, \quad \frac{1}{\upsilon_1}= \frac{\alpha+1}{q_1}, \quad \theta \in (0,\alpha), \quad \frac{1}{p_1'}=\frac{1+\theta_1}{p_1}.
\]
This implies
\begin{align}
\frac{d}{\gamma_1}=d-\frac{d(\alpha+2)}{q_1}>b, \quad p_1 =\theta_1+2 \in (2, \alpha+2). \label{condition on B}
\end{align}
The first condition in $(\ref{condition on B})$ is equivalent to $q_1>\frac{d(\alpha+2)}{d-b}$. Let us choose
\begin{align}
q_1=\frac{d(\alpha+2)}{d-b}+\ep, \label{choice q_1}
\end{align}
for some $0<\ep\ll 1$ to be chosen later. Since $\alpha_\star<\alpha<\alpha^\star$, by taking $\ep>0$ sufficiently small, it is easy to see that $q_1 \in (2, 2^\star)$. It remains to check $p_1<\alpha+2$. Since $(p_1, q_1)\in S$, we need to show
\[
\frac{2}{p_1}=\frac{d}{2}-\frac{d}{q_1}>\frac{2}{\alpha+2} \quad \text{or}\quad \frac{d}{q_1}<\frac{d(\alpha+2)-4}{2(\alpha+2)}.
\]
It is in turn equivalent to
\[
d(\alpha+2)(d\alpha-4+2b) + \ep(d-b)[d(\alpha+2)-4]>0.
\]
Since $\alpha>\alpha_\star=\frac{4-2b}{d}$, the above inequality holds true by taking $\ep>0$ small enough. We thus obtain
\begin{align}
A_1 \lesssim \|u\|^{\alpha-\theta_1}_{L^\infty_t(J, L^{q_1}_x)} \|u\|^{1+\theta_1}_{L^{p_1}_t(J, L^{q_1}_x)} \lesssim \|u\|^{\alpha-\theta_1}_{L^\infty_t(J, L^{q_1}_x)} \|u\|^{1+\theta_1}_{S(J)}. \label{scattering estimate proof 1}
\end{align}
We next bound
\[
B_1 \leq \||x|^{-b}\nabla(|u|^\alpha u)\|_{S'(L^2(B), J)} + \||x|^{-b-1}|u|^\alpha u\|_{S'(L^2(B), J)} =:B_{11}+ B_{12}.
\]
By the fractional chain rule, we estimate
\begin{align*}
B_{11} \leq \||x|^{-b}\nabla(|u|^\alpha u)\|_{L^{p_1'}_t(J, L^{q_1'}_x(B))} &\lesssim \||x|^{-b}\|_{L^{\gamma_1}_x(B)}\|\nabla(|u|^\alpha u)\|_{L^{p_1'}_t(J, L^{\upsilon_1}_x)} \\
&\lesssim \|\|u\|^\alpha_{L^{q_1}_x} \|\nabla u\|_{L^{q_1}_x}\|_{L^{p_1'}_t(J)} \\
&\lesssim \|u\|^{\alpha-\theta_1}_{L^\infty_t(J, L^{q_1}_x)} \|u\|^{\theta_1}_{L^{p_1}_t(J, L^{q_1}_x)} \|\nabla u\|_{L^{p_1}_t(J, L^{q_1}_x)},
\end{align*}
provided that $(p_1, q_1)\in S$ and 
\[
\frac{1}{q_1'}=\frac{1}{\gamma_1}+\frac{1}{\upsilon_1}, \quad \frac{d}{\gamma_1}>b, \quad \frac{1}{\upsilon_1}=\frac{\alpha+1}{q_1}, \quad \frac{1}{p_1'}=\frac{1+\theta_1}{q_1}, \quad \theta_1 \in (0, \alpha).
\]
This implies
\[
\frac{d}{\gamma_1}=d-\frac{d(\alpha+2)}{q_1}>b, \quad p_1=\theta_1 + 2 \in (2, \alpha+2).
\]
This condition is exactly $(\ref{condition on B})$. Therefore, we choose $q_1$ as in $(\ref{choice q_1})$ and get
\begin{align}
B_{11} \lesssim \|u\|^{\alpha-\theta_1}_{L^\infty_t(J, L^{q_1}_x)} \|u\|^{\theta_1}_{L^{p_1}_t(J, L^{q_1}_x)} \|\nabla u\|_{L^{p_1}_t(J, L^{q_1}_x)} \lesssim \|u\|^{\alpha-\theta_1}_{L^\infty_t(J, L^{q_1}_x)} \|u\|^{1+\theta_1}_{S(J)}. \label{scattering estimate proof 2} 
\end{align}
We next bound
\begin{align*}
B_{12} \leq \||x|^{-b-1} |u|^\alpha u\|_{L^{p_1'}_t(J, L^{q_1'}_x(B))} &\lesssim \||x|^{-b}\|_{L^{\gamma_1}_x(B)} \||u|^\alpha u\|_{L^{p_1'}_t(J, L^{\upsilon_1}_x)} \\
&\lesssim \|\|u\|^\alpha_{L^{q_1}_x} \|u\|_{L^{n_1}_x}\|_{L^{p_1'}_t(J)} \\
&\lesssim \|u\|^{\alpha-\theta_1}_{L^\infty_t(J, L^{q_1}_x)} \|u\|^{\theta_1}_{L^{p_1}_t(J, L^{q_1}_x)} \|u\|_{L^{p_1}_t(J, L^{n_1}_x)}.
\end{align*}
\indent \underline{When $d\geq 4$}, we use the homogeneous Sobolev embedding $\|u\|_{L^{n_1}_x} \lesssim \|\nabla u\|_{L^{q_1}_x}$ to have
\[
B_{12}\lesssim \|u\|^{\alpha-\theta_1}_{L^\infty_t(J, L^{q_1}_x)} \|u\|^{\theta_1}_{L^{p_1}_t(J, L^{q_1}_x)} \|\nabla u\|_{L^{p_1}_t(J, L^{q_1}_x)}.
\]
The above estimates hold true provided that $(p_1, q_1)\in S$ and
\[
\frac{1}{q_1'}=\frac{1}{\gamma_1} + \frac{1}{\upsilon_1}, \quad \frac{d}{\gamma_1}>b+1, \quad \frac{1}{\upsilon_1}=\frac{\alpha}{q_1}+\frac{1}{n_1}, \quad \frac{1}{p_1'}=\frac{1+\theta_1}{p_1}, \quad \theta_1 \in (0,\alpha),
\]
and 
\[
q_1 <d, \quad \frac{1}{n_1}=\frac{1}{q_1}-\frac{1}{d}.
\]
Note that the last condition allows us to use the homogeneous Sobolev embedding. The above requirements imply
\[
\frac{d}{\gamma_1}= d-\frac{d(\alpha+2)}{q_1}+1>b+1, \quad p_1 =\theta_1+2 \in(2, \alpha+2).
\]
This is exactly $(\ref{condition on B})$. We thus choose $q_1$ as in $(\ref{choice q_1})$. Note that by taking $\ep>0$ small enough, the requirement $q_1<d$ is satisfied if 
\begin{align}
\frac{d(\alpha+2)}{d-b}<d \quad \text{or}\quad \alpha< d-b-2. \label{requirement q<d}
\end{align}
Since $d\geq 4$, it is easy to check that $\alpha^\star=\frac{4-2b}{d-2} \leq d-b-2$. We thus get for $d\geq 4, 0<b<2$ and $\alpha_\star<\alpha<\alpha^\star$,
\begin{align}
B_{12} \lesssim \|u\|^{\alpha-\theta_1}_{L^\infty_t(J, L^{q_1}_x)} \|u\|^{1+\theta_1}_{S(J)}. \label{scattering estimate proof 3}
\end{align}
\indent \underline{When $d=3$}, we firstly note that $(\ref{requirement q<d})$ does not hold true. We use instead the inhomogeneous Sobolev embedding $\|u\|_{L^{n_1}_x} \lesssim \|\scal{\nabla} u\|_{L^{q_1}_x}$ to have
\[
B_{12} \lesssim \|u\|^{\alpha-\theta_1}_{L^\infty_t(J, L^{q_1}_x)} \|u\|^{\theta_1}_{L^{p_1}_t(J, L^{q_1}_x)} \|\scal{\nabla} u\|_{L^{p_1}_t(J, L^{q_1}_x)}.
\]
The above estimate holds true provided that $(p_1, q_1)\in S$ and 
\[
\frac{1}{q_1'}=\frac{1}{\gamma_1} + \frac{1}{\upsilon_1}, \quad \frac{3}{\gamma_1}>b+1, \quad \frac{1}{\upsilon_1}=\frac{\alpha}{q_1}+\frac{1}{n_1}, \quad \frac{1}{p_1'}=\frac{1+\theta_1}{p_1}, \quad \theta_1 \in (0,\alpha),
\]
and 
\[
3< q_1 , \quad n_1 \in (q_1, \infty) \quad \text{or}\quad \frac{1}{n_1}=\frac{\tau}{q_1}, \quad \tau \in (0,1).
\]
Here the last condition ensures the inhomogeneous Sobolev embedding. The above requirements imply
\[
\frac{3}{\gamma_1}= 3-\frac{3(\alpha+1+\tau)}{q_1}>b-1 \quad \text{or} \quad \frac{3(\alpha+1+\tau)}{q_1}<2-b. 
\] 
Let us choose
\begin{align*}
q_1=\frac{3(\alpha+1+\tau)}{2-b}+\ep, 
\end{align*}
for some $0<\ep\ll 1$ to be chosen later. It remains to check 
\begin{align*}
q_1 \in (3,6), \quad p_1 \in (2, \alpha+2). 
\end{align*}
By taking $\ep>0$ small enough, the condition $q_1 \in (3,6)$ implies
\begin{align}
1-b-\tau <\alpha <3-2b-\tau. \label{condition alpha 1}
\end{align}
Since $(p_1, q_1)\in S$, the condition \footnote{Note that $q_1<6$ implies $p_1>2$.} $p_1<\alpha+2$ is equivalent to
\[
\frac{3}{2}-\frac{3}{q_1}=\frac{2}{p_1}>\frac{2}{\alpha+2}.
\]
The above condition is then equivalent to
\[
3[3\alpha^2 + (1+2b)\alpha +4b-6 + \tau (3\alpha+2)] +\ep(2-b)(3\alpha+2)>0.
\]
By taking $\ep>0$ sufficiently small, the above inequality holds true provided that
\begin{align}
3\alpha^2 + (1+2b)\alpha +4b-6 + \tau (3\alpha+2)>0. \label{condition alpha 2}
\end{align}
Now, if we take $\tau$ closed to $0$, $(\ref{condition alpha 1})$ and $(\ref{condition alpha 2})$ imply
\[
1-b<\alpha<3-2b, \quad \alpha>\frac{-1-2b+\sqrt{4b^2-44b+73}}{6}.
\]
Combining this with the assumption $\frac{4-2b}{3}=\alpha_\star<\alpha <\alpha^\star=4-2b$, we have
\begin{align}
\frac{4-2b}{3}<\alpha <3-2b, \quad 0<b<\frac{5}{4}. \label{condition alpha b}
\end{align}
We thus obtain for $d=3$ and $\alpha, b$ as in $(\ref{condition alpha b})$,
\begin{align}
B_{12} \lesssim \|u\|^{\alpha-\theta_1}_{L^\infty_t(J, L^{q_1}_x)} \|u\|^{1+\theta_1}_{S(J)}. \label{scattering estimate proof 4}
\end{align}
\indent \underline{\bf On $B^c$.} By H\"older's inequality and Remark $\ref{rem dealing singularity}$,
\begin{align*}
A_2 \leq \||x|^{-b}|u|^\alpha u\|_{L^{p_2'}_t(J, L^{q_2'}_x(B^c))} &\lesssim \||x|^{-b}\|_{L^{\gamma_2}_x(B^c)} \||u|^\alpha u\|_{L^{p_2'}_t(J, L^{\upsilon_2}_x)} \\
&\lesssim \|\|u\|^{\alpha+1}_{L^{q_2}_x}\|_{L^{p_2'}_t(J)} \\
&\lesssim \|u\|^{\alpha-\theta_2}_{L^\infty_t(J, L^{q_2}_x)} \|u\|^{1+\theta_2}_{L^{p_2}_t(J, L^{q_2}_x)},
\end{align*}
provided that $(p_2, q_2)\in S$ and
\[
\frac{1}{q_2'}=\frac{1}{\gamma_2}+\frac{1}{\upsilon_2}, \quad \frac{d}{\gamma_2}<b, \quad \frac{1}{\upsilon_2}=\frac{\alpha+1}{q_2}, \quad \frac{1}{p_2'}=\frac{1+\theta_2}{p_2}, \quad \theta_2 \in (0,\alpha).
\]
This implies
\begin{align}
\frac{d}{\gamma_2}=d-\frac{d(\alpha+2)}{q_2} <b, \quad p_2 =\theta_2+2 \in (2, \alpha+2). \label{condition on Bc}
\end{align}
The first condition in $(\ref{condition on Bc})$ implies $q_2 <\frac{d(\alpha+2)}{d-b}$. Let us choose
\begin{align}
q_2=\frac{d(\alpha+2)}{d-b} -\ep, \label{choice q_2}
\end{align}
for some $0<\ep\ll 1$ to be chosen later. By taking $\ep>0$ small enough, the assumption $\alpha_\star<\alpha<\alpha^\star$ ensures $q_2 \in (2, 2^\star)$. It remains to check $p_2 <\alpha+2$. Since $(p_2, q_2)\in S$, it is equivalent to
\[
\frac{d}{2}-\frac{d}{q_2}=\frac{2}{p_2}>\frac{2}{\alpha+2} \quad \text{or}\quad \frac{d}{q_2}<\frac{d(\alpha+2)-4}{2(\alpha+2)}.
\]
A direct computation shows that the above condition is equivalent to
\[
d(\alpha+2)(d\alpha-4+2b) -\ep (d-b) [d(\alpha+2)-4] >0.
\]
Since $\alpha>\frac{4-2b}{d}$, the above inequality holds true by taking $\ep>0$ small enough. We thus get
\begin{align}
A_2 \lesssim \|u\|^{\alpha-\theta_2}_{L^\infty_t(J, L^{q_2}_x)} \|u\|^{1+\theta_2}_{L^{p_2}_t(J, L^{q_2}_x)} \lesssim \|u\|^{\alpha-\theta_2}_{L^\infty_t(J, L^{q_2}_x)} \|u\|^{1+\theta_2}_{S(J)}. \label{scattering estimate proof 5}
\end{align}
We next bound
\[
B_2 \leq \||x|^{-b}\nabla(|u|^\alpha u)\|_{S'(L^2(B^c), J)} + \||x|^{-b-1}|u|^\alpha u\|_{S'(L^2(B^c), J)}=: B_{21}+B_{22}.
\]
By the fractional chain rule, H\"older's inequality and Remark $\ref{rem dealing singularity}$, 
\begin{align*}
B_{21}\leq \||x|^{-b}\nabla(|u|^\alpha u)\|_{L^{p_2'}_t(J, L^{q_2'}_x(B^c))} &\lesssim \||x|^{-b}\|_{L^{\gamma_2}_x(B^c)} \|\nabla(|u|^\alpha u)\|_{L^{p_2'}_t(J, L^{\upsilon_2}_x)} \\
&\lesssim \|\|u\|^\alpha_{L^{q_2}_x} \|\nabla u\|_{L^{q_2}_x}\|_{L^{p_2'}_t(J)} \\
&\lesssim \|u\|^{\alpha-\theta_2}_{L^\infty_t(J, L^{q_2}_x)} \|u\|^{\theta_2}_{L^{p_2}_t(J, L^{q_2}_x)} \|\nabla u\|_{L^{p_2}_t(J, L^{q_2}_x)},
\end{align*}
provided that $(p_2,q_2)\in S$ and 
\[
\frac{1}{q_2'}=\frac{1}{\gamma_2}+\frac{1}{\upsilon_2}, \quad \frac{d}{\gamma_2}<b, \quad \frac{1}{\upsilon_2}=\frac{\alpha+1}{q_2}, \quad \frac{1}{p_2'}=\frac{1+\theta_2}{p_2}, \quad \theta_2 \in (0,\alpha).
\]
These conditions are exactly those for $A_2$. We thus choose $q_2$ as in $(\ref{choice q_2})$ and obtain
\begin{align}
B_{21} \lesssim \|u\|^{\alpha-\theta_2}_{L^\infty_t(J, L^{q_2}_x)} \|u\|^{\theta_2}_{L^{p_2}_t(J, L^{q_2}_x)} \|\nabla u\|_{L^{p_2}_t(J, L^{q_2}_x)} \lesssim \|u\|^{\alpha-\theta_2}_{L^\infty_t(J, L^{q_2}_x)} \|u\|^{1+\theta_2}_{S(J)}. \label{scattering estimate proof 6} 
\end{align}
It remains to treat $B_{22}$. By H\"older's inequality and Remark $\ref{rem dealing singularity}$,
\begin{align*}
B_{22}\leq \||x|^{-b-1}|u|^\alpha u\|_{L^{p_2'}_t(J, L^{q_2'}_x(B^c))} &\lesssim \||x|^{-b-1}\|_{L^{\gamma_2}_x(B^c)} \||u|^\alpha u\|_{L^{p_2'}_t (J, L^{\upsilon_2}_x)} \\
&\lesssim \|\|u\|^\alpha_{L^{q_2}_x} \|u\|_{L^{n_2}_x}\|_{L^{p_2'}_t(J)} \\
&\lesssim \|u\|^{\alpha-\theta_2}_{L^\infty_t(J, L^{q_2}_x)} \|u\|^{\theta_2}_{L^{p_2}_t(J, L^{q_2}_x)} \|u\|_{L^{p_2}_t(J, L^{n_2}_x)},
\end{align*}
provided that
\begin{align}
\frac{1}{q_2'}=\frac{1}{\gamma_2}+\frac{1}{\upsilon_2}, \quad \frac{d}{\gamma_2}<b+1,\quad \frac{1}{\upsilon_2}=\frac{\alpha}{q_2} +\frac{1}{n_2}, \quad \frac{1}{p_2'}=\frac{1+\theta_2}{p_2}, \quad \theta_2 \in (0,\alpha). \label{condition B22}
\end{align}
As for $B_{12}$, we separate two cases: $d\geq 4$ and $d=3$. \newline
\indent \underline{When $d\geq 4$}, we use the homogeneous Sobolev embedding $\|u\|_{L^{n_2}_x} \lesssim \|\nabla u\|_{L^{q_1}_x}$ provided that
\[
q_2<d, \quad \frac{1}{n_2}=\frac{1}{q_2}-\frac{1}{d}. 
\]
Thus, $(\ref{condition B22})$ implies
\[
\frac{d}{\gamma_2}=d-\frac{d(\alpha+2)}{q_2}+1 <b+1 \quad \text{or} \quad d-\frac{d(\alpha+2)}{q_2}<b, \quad p_2 =\theta_2 +2 \in (2, \alpha+2).
\]
This condition is exactly $(\ref{condition on Bc})$. We thus choose $q_2$ as in $(\ref{choice q_2})$. Note that by taking $\ep>0$ small enough, this condition holds true if we have
\[
\frac{d(\alpha+2)}{d-b}<d \quad \text{or} \quad \alpha<d-b-2.
\]
Since $d\geq 4$, we always have $\frac{4-2b}{d-2}<d-b-2$. Therefore, the last estimate holds true for $\alpha_\star<\alpha<\alpha^\star$. We obtain for $d\geq 4, 0<b<2$ and $\alpha_\star<\alpha<\alpha^\star$,
\begin{align}
B_{22} \lesssim \|u\|^{\alpha-\theta_2}_{L^\infty_t(J, L^{q_2}_x)} \|u\|^{\theta_2}_{L^{p_2}_t(J, L^{q_2}_x)} \|\nabla u\|_{L^{p_2}_t(J, L^{q_2}_x)} \lesssim \|u\|^{\alpha-\theta_2}_{L^\infty_t(J, L^{q_2}_x)} \|u\|^{1+\theta_2}_{S(J)}. \label{scattering estimate proof 7}
\end{align}
\indent \underline{When $d=3$}, we use the inhomogeneous Sobolev embedding $\|u\|_{L^{n_2}_x} \lesssim \|\scal{\nabla} u\|_{L^{q_2}_x}$ provided that
\[
q_2>3, \quad n_2 \in (q_2, \infty)\quad \text{or} \quad \frac{1}{n_2}=\frac{\tau}{q_2}, \quad \tau \in (0,1).
\]
Thus, $(\ref{condition B22})$ implies
\[
\frac{3}{\gamma_2}=3-\frac{3(\alpha+1+\tau)}{q_2}<b+1 \quad \text{or}\quad \frac{3(\alpha+1+\tau)}{q_2}>2-b.
\]
Let us choose 
\[
q_2=\frac{3(\alpha+1+\tau)}{2-b}-\ep,
\]
for some $0<\ep\ll 1$ to be chosen later. We need to check $q_2 \in (3,6)$ and $p_2 \in (2,\alpha+2)$. By taking $\ep>0$ sufficiently small, these conditions hold true if we have
\begin{align*}
1-b-\tau <\alpha<3-2b-\tau, \quad 3\alpha^2 +(1+2b)\alpha + 4b-6 + \tau (3\alpha+2)>0.
\end{align*}
Taking $\tau$ closed to $0$, we have
\[
1-b<\alpha<3-2b, \quad \alpha>\frac{-1-2b+\sqrt{4b^2-44b+73}}{6}. 
\]
By the assumption $\frac{4-2b}{3}<\alpha<4-2b$, we see that $b$ and $\alpha$ satisfy $(\ref{condition alpha b})$. Therefore, we get for $d=3$ and $b, \alpha$ as in $(\ref{condition alpha b})$,
\begin{align}
B_{22}\lesssim \|u\|^{\alpha-\theta_2}_{L^\infty_t(J, L^{q_2}_x)} \|u\|^{\theta_2}_{L^{p_2}_t(J, L^{q_2}_x)} \|\scal{\nabla} u\|_{L^{p_2}_t(J, L^{q_2}_x)} \lesssim \|u\|^{\alpha-\theta_2}_{L^\infty_t(J, L^{q_2}_x)} \|u\|^{1+\theta_2}_{S(J)}. \label{scattering estimate proof 8}
\end{align} 
Collecting $(\ref{scattering estimate proof 1}), (\ref{scattering estimate proof 2}), (\ref{scattering estimate proof 3}), (\ref{scattering estimate proof 4}), (\ref{scattering estimate proof 5}), (\ref{scattering estimate proof 6}), (\ref{scattering estimate proof 7})$ and  $(\ref{scattering estimate proof 8})$, we complete the proof.
\end{proof}
\begin{coro} \label{coro global bound}
Let $d, b$ and $\alpha$ be as in Theorem $\ref{theorem energy scattering}$. Let $u_0 \in H^1(\R^d)$ and $u$ be the unique global solution to the defocusing \emph{(INLS)}. Then 
\[
u \in L^p_t(\R, W^{1, q}_x),
\]
for any Schr\"odinger admissible pair $(p,q)$. 
\end{coro}
\begin{proof}
By applying Lemma $\ref{lem decay property approach}$ with $J=(T, t)$ and using the decay property given in Theorem $\ref{theorem decay property}$, we see that there exist $\theta_1, \theta_2  \in (0,\alpha)$ such that
\[
\|u\|_{S((T, t))} \lesssim \|u(T)\|_{H^1_x} + \ep_1(T) \|u\|^{1+\theta_1}_{S((T, t))} + \ep_2(T) \|u\|^{1+\theta_2}_{S((T, t))},
\] 
where $\ep_1(T), \ep_2(T) \rightarrow 0$ as $T\rightarrow +\infty$. By the conservations law and the continuity argument (see e.g. \cite[Lemma 7.7.4]{Cazenave} or \cite[Section 1.3]{Tao}), we learn that for $T$ large enough, 
\[
\|u\|_{S((T, t))} \leq C,
\]
for some $C>0$ independent of $t$. We thus get $u \in L^p_t((T,+\infty), W^{1,q}_x)$ for any $(p,q)\in S$. Similarly, we prove as well that $u \in  L^p_t((-\infty,-T), W^{1,q}_x)$ for any $(p,q)\in S$. Combining these facts and the local well-posedness given in Theorem $\ref{theorem dinhweighted}$, we obtain $u \in L^p_t(\R, W^{1, q}_x)$ for any Schr\"odinger admissible pair $(p,q)$. 
\end{proof}
We are now able to prove Theorem $\ref{theorem energy scattering}$. The proof is based on a standard argument (see e.g. \cite[Section 8.3]{Cazenave} or \cite[Section 3.6]{Tao}).
\paragraph{\bf Proof of Theorem $\ref{theorem energy scattering}$.}
Let $u$ be the global solution to the defocusing (INLS). By Duhamel's formula, 
\begin{align}
u(t)= e^{it\Delta} u_0 -i \int_0^t e^{i(t-s)\Delta} |x|^{-b}|u(s)|^\alpha u(s) ds. \label{duhamel formula scattering}
\end{align}
As in the proof of Lemma $\ref{lem decay property approach}$, we see that there exists $\theta_1, \theta_2 \in (0,\alpha)$ and $q_1, q_2 \in (2, 2^\star)$ such that
\[
\|\scal{\nabla}(|x|^{-b} |u|^\alpha u)\|_{S'(L^2, \R)} \lesssim \|u\|^{\alpha-\theta_1}_{L^\infty_t(\R, L^{q_1}_x)} \|u\|_{S(\R)}^{1+\theta_1}+ \|u\|^{\alpha-\theta_2}_{L^\infty_t(\R, L^{q_2}_x)} \|u\|_{S(\R)}^{1+\theta_2}.
\]
Thus, Theorem $\ref{theorem decay property}$ and Corollary $\ref{coro global bound}$ imply
\begin{align}
\|\scal{\nabla}(|x|^{-b} |u|^\alpha u)\|_{S'(L^2, \R)} <\infty. \label{global bound}
\end{align}
Let $0<t_1<t_2<+\infty$. By Strichartz estimates and $(\ref{global bound})$, we have 
\[
\|e^{-it_2\Delta} u(t_2) - e^{-it_1 \Delta} u(t_1)\|_{H^1_x} \lesssim \|\scal{\nabla}(|x|^{-b} |u|^\alpha u)\|_{S'(L^2, (t_1, t_2))} \rightarrow 0, 
\]
as $t_1, t_2\rightarrow +\infty$. This implies that the limit
\[
u^+_0:=\lim_{t\rightarrow +\infty} e^{-it\Delta} u(t)
\]
exists in $H^1_x$. Moreover, 
\[
u(t) - e^{it\Delta} u^+_0 = -i\int_t^{+\infty} e^{i(t-s)\Delta} |x|^{-b} |u(s)|^\alpha u(s) ds.
\]
Applying again Strichartz estimates, we get
\[
\lim_{t\rightarrow +\infty} \|u(t) -e^{it\Delta} u^+_0\|_{H^1_x} =0.
\]
This gives the result for positive time, the one for negative time is treated similarly. The proof is complete.
\defendproof

\section*{Appendix}
\setcounter{equation}{0}
In this Appendix, we will give the proof of $(\ref{momentum bracket})$. Let $N(u)=F(x,u)=|x|^{-b}|u|^\alpha u$. We compute
\begin{align*}
\nabla(|x|^{-b}|u|^{\alpha+2}) &= \nabla(|x|^{-b}) |u|^{\alpha+2} + |x|^{-b} \nabla(|u|^{\alpha+2}) \\
&= \nabla(|x|^{-b}) |u|^{\alpha+2} + (\alpha+2)|x|^{-b} |u|^\alpha \re{(\nabla u \overline{u})} \\
&= \nabla(|x|^{-b}) |u|^{\alpha+2} + (\alpha+2)|x|^{-b} |u|^\alpha \re{(u \nabla \overline{u})}.
\end{align*}
Similarly,
\[
\nabla(|x|^{-b}|u|^{\alpha+2}) = \nabla(|x|^{-b}|u|^\alpha \overline{u} u) = \nabla(|x|^{-b}|u|^\alpha \overline{u}) u + |x|^{-b}|u|^\alpha \overline{u} \nabla u,
\]
or
\[
\nabla(|x|^{-b}|u|^\alpha \overline{u}) u = \nabla(|x|^{-b}|u|^{\alpha+2}) - |x|^{-b}|u|^\alpha \overline{u} \nabla u.
\]
Therefore,
\begin{align*}
\{N(u), u\}_p &= \re{(|x|^{-b}|u|^\alpha u \nabla \overline{u} - u \nabla(|x|^{-b}|u|^\alpha \overline{u}))} \\
&=2\re{(|x|^{-b}|u|^\alpha u \nabla \overline{u})} -\nabla(|x|^{-b}|u|^{\alpha+2}) \\
&= \frac{2}{\alpha+2}\left(\nabla(|x|^{-b}|u|^{\alpha+2}) -\nabla(|x|^{-b}) |u|^{\alpha+2} \right) -\nabla(|x|^{-b}|u|^{\alpha+2}) \\
&= -\frac{\alpha}{\alpha+2} \nabla(|x|^{-b}|u|^{\alpha+2}) -\frac{2}{\alpha+2} \nabla(|x|^{-b}) |u|^{\alpha+2}.
\end{align*}

\section*{Acknowledgments}
The author would like to express his deep thanks to his wife-Uyen Cong for her encouragement and support. He would like to thank his supervisor Prof. Jean-Marc Bouclet for the kind guidance and constant encouragement. He also would like to thank the reviewer for his/her helpful comments and suggestions. 


\end{document}